[1994/12/01]
\documentclass[11pt,final]{amsart}
\usepackage{enumitem}
\usepackage{amssymb}
\usepackage{mathtools}
\usepackage{fullpage}
\usepackage{pifont}
\chardef\bslash=`\\ 

\usepackage{enumitem}
\usepackage{amssymb}
\usepackage{mathtools}
\usepackage{pifont}

\theoremstyle{plain}
\newtheorem{theorem}{Theorem}[section]
\newtheorem{lemma}[theorem]{Lemma}
\newtheorem{proposition}[theorem]{Proposition}

\newtheorem{conjecture}[theorem]{Conjecture}

\theoremstyle{remark}

\newtheorem{example}{Example}

\theoremstyle{definition}
\newtheorem{definition}[theorem]{Definition}

\newcommand{\Z}{\mathbb{Z}}
\newcommand{\Q}{\mathbb{Q}}

\DeclareMathOperator{\Aut}{Aut}
\DeclareMathOperator{\GL}{GL}

\DeclareMathOperator{\supp}{supp}

\begin{document}
\title[Free symmetric and unitary pairs]{Free symmetric and unitary pairs in the field 
of fractions of torsion-free nilpotent group algebras}

\author[V.~O.~Ferreira]{Vitor O. Ferreira}
\address{Department of Mathematics, University of S\~{a}o Paulo, S\~{a}o Paulo, SP, 
05508-090, Brazil}
\email{vofer@ime.usp.br}
\thanks{The first author was partially supported by FAPESP-Brazil, 
Proj.~Tem\'atico 2015/09162-9.}

\author[J.~Z.~Gon\c calves]{Jairo Z. Gon\c calves}
\address{Department of Mathematics, University of S\~{a}o Paulo, S\~{a}o Paulo, SP, 
05508-090, Brazil}
\email{jz.goncalves@usp.br}
\thanks{The second author was partially supported by FAPESP-Brazil, 
Proj.~Tem\'atico 2015/09162-9, and CNPq, Grant 301205/2015-9.}

\author[J. S\'anchez]{Javier S\'anchez}
\address{Department of Mathematics, University of S\~{a}o Paulo, S\~{a}o Paulo, SP, 
05508-090, Brazil}
\email{jsanchez@ime.usp.br}
\thanks{The third author was partially supported by FAPESP-Brazil, 
Proj.~Tem\'atico 2015/09162-9, and CNPq, Grant 307638/2015-4.}



\date{21 May 2018}

\begin{abstract} Let $k$ be a field of characteristic different from $2$ and let
$G$ be a nonabelian residually torsion-free nilpotent group. It is known that
$G$ is an orderable group. Let $k(G)$ denote 
the subdivision ring of the Malcev-Neumann series ring generated by the group 
algebra of $G$ over $k$. If $\ast$ is an involution on $G$, then it extends to a unique $k$-involution
on $k(G)$. We show that $k(G)$ contains pairs of symmetric 
elements with respect to $\ast$ which generate a free group inside the 
multiplicative group of $k(G)$. Free unitary pairs also exist if $G$ is 
torsion-free nilpotent. Finally, we consider the general case of a division ring $D$, with 
a $k$-involution $\ast$, containing a normal subgroup $N$ in its multiplicative group,
such  that $G \subseteq N$, with $G$ a nilpotent-by-finite torsion-free subgroup 
that is not abelian-by-finite, satisfying $G^{*}=G$ and $N^{*}=N$. We prove that $N$ 
contains a free symmetric pair.
\end{abstract}
\maketitle

\section{Introduction}\label{S:Intro}

In \cite{Lichtman77}, Lichtman raises the interesting

\begin{conjecture}\label{C:LichtConj} The multiplicative group of a noncommutative 
division ring contains a free noncyclic subgroup.
\end{conjecture}

Frequently division rings come accompanied by involutions. One instance is the case of 
division rings of fractions of group algebras which are Ore domains. 

To be more precise, by an \emph{involution} on a group $G$ one understands an 
anti-automorphism of $G$ of order $2$. If $k$ is a field a $k$-\emph{involution} on a $k$-algebra
$R$ is a $k$-linear map $\ast$ satisfying $(xy)^{\ast}=y^{\ast}x^{\ast}$ and $(x^{\ast})^{\ast}=x$
for all $x,y\in R$. Thus, if $G$ is a group and $k$ is a field, the $k$-linear extension of
an involution on $G$ to the group algebra $k[G]$ is a $k$-involution. Suppose that $k[G]$ is embedded
in a division ring $D$ generated (as a division ring) by $k[G]$. There is at most one involution on $D$ which is the
extension of the involution on $k[G]$. When this is the case, we shall say that the $k$-involution on $D$ is \emph{induced}
by the involution of $G$. For example, if $k[G]$ is an Ore domain with division ring of fractions $D$, there always exists
a $k$-involution on $D$ induced by the involution of $G$.

An element $x$ of an algebra $R$ with an involution $\ast$ is called \emph{symmetric} if 
$x^{\ast}=x$ and \emph{unitary} if 
$xx^{\ast}=x^{\ast}x=1$. And if $x$ and $y$ are elements of $R$, we say that they form a 
\emph{free symmetric (resp.~unitary) pair} if both are symmetric 
(resp.~unitary) and generate a free noncyclic subgroup
of the group of units of $R$. We shall look into the existence of free symmetric
and unitary pairs in division rings with involution. In what follows, if $D$ is
a division ring, we shall denote the multiplicative group $D\setminus\{0\}$ of $D$
by $D^{\bullet}$.

In \cite{GS06}, an analogue of Conjecture~\ref{C:LichtConj} was proved for 
finite-dimensional division rings with 
involutions, namely: let $D$ be a division algebra over a field $k$ of characteristic
different from $2$ and let $\ast$ be a $k$-involution on $D$. If $D$ is finite
dimensional over its center, then it contains free symmetric and 
unitary pairs, unless it is a quaternion algebra with an appropriate involution in 
each case. 
In \cite{FG15}, \cite{GS12} and \cite{G17}, the existence of free symmetric and unitary 
pairs when $D$ is a division ring with involution, infinite-dimensional over its center,
$k$ was 
considered.

In particular, in \cite{G17}, it was shown that for the division ring of fractions of the group 
algebra $k[\Gamma]$ of the \emph{Heisenberg group} 
\[
\Gamma = \langle x, y : [x,[x,y]]=[y,[x,y]]=1 \rangle,
\]
where $[x,y]$ stands for the commutator $x^{-1}y^{-1}xy$, the following holds.

\begin{theorem}[{\cite[Theorem~2]{G17}}]\label{T:InvolHeisenb} Let $D$ be the division
ring of 
fractions of the group algebra $k[\Gamma]$ of the Heisenberg group $\Gamma$ over a field 
$k$ of characteristic different from $2$, and let $\ast$ be a $k$-involution of $D$ induced 
from an involution of $\Gamma$. Then $D$ contains both free symmetric and free unitary 
pairs with respect to $\ast$.
\end{theorem}

Our objective in Section~\ref{SN} is, following the ideas of \cite{FG15}, to present a proof of

\begin{theorem}\label{T:Main} Let $k[G]$ be the group algebra of a nonabelian torsion-free
nilpotent group $G$ over a field $k$ of characteristic different from $2$ and let $D$ be the division
ring of fractions of $k[G]$. Let $\ast$ be a $k$-involution on $D$ which is induced by an 
involution of $G$.
Then $D$ contains both free symmetric and free unitary pairs with respect to $\ast$.
\end{theorem}

Regarding the existence of free symmetric pairs, we can also handle the case of normal
subgroups of division rings containing a torsion-free nilpotent-by-finite subgroup not abelian-by-finite, as the following result shows. We note that, by \cite[1.3.2]{LR04}, the class
of nilpotent-by-finite groups contains the class of supersoluble groups.

\begin{theorem}\label{T:SubMain} Let $D$ be a division ring with center $k$ of
characteristic different from $2$, with a
$k$-involution $\ast$. Let $N$ be a normal subgroup of $D^{\bullet}$ and assume that
$N$ contains a torsion-free nilpotent-by-finite subgroup $G$ which is not abelian-by-finite. 
If both $G$ and $N$ are invariant under $\ast$, then $N$ contains a free symmetric pair
with respect to $\ast$.
\end{theorem}

One can say that Theorem~\ref{T:SubMain} is an involutional version, for symmetric 
elements, of the following result of Lichtman.

\begin{theorem}[{\cite[Theorem~2]{Lichtman78}}]\label{T:LichtmanNilpInNormal} Let 
$D$ be a division ring and let $N$ be a normal subgroup of $D^{\bullet}$.
Suppose that there exists a nonabelian nilpotent-by-finite subgroup $G$ of $N$. Then $N$ contains a noncyclic free subgroup.
\end{theorem}

Unfortunately, we do not have a complete analogue of this result, since we are still 
unable to handle the case $N \vartriangleleft D^{\bullet}$ when $D$ is finite-dimensional 
over its center $k$. This happens precisely when $G$ is abelian-by-finite.

The restriction on the characteristic of $k$ is usual, since many strange things happen 
for involutions when the characteristic is $2$. For instance, our technique to construct 
specializations from  the division ring of fractions of the Heisenberg group algebra to a 
quaternion algebra collapses in characteristic $2$. 

Section~\ref{SR} presents a completion argument that allows the main result on free
symmetric pairs
of Section~\ref{SN} to be extended to residually torsion-free nilpotent groups. More precisely, 
given a field $k$, a residually torsion-free nilpotent group $G$ and a total ordering 
$<$ on $G$ such
that $(G,<)$ is an ordered group, let $k(G)$ denote the subdivision ring of the
Malcev-Neumann series ring $k((G,<))$ generated by the group algebra $k[G]$. Let
$\ast$ be a $k$-involution on $G$ which is induced by an involution of $G$.
In \cite[Theorem~2.9]{FGS13}, it was proved that $\ast$ has a unique extension 
to a $k$-involution of $k(G)$. The main result of Section~\ref{SR} is the following.

\begin{theorem}\label{th:restorfrnil}
Let $k$ be a field of characteristic different from $2$ and let $G$ be a nonabelian residually 
torsion-free nilpotent group with an involution $\ast$. Then $k(G)$ contains a free
symmetric pair with respect to the $k$-involution induced by $\ast$. 
\end{theorem} 

\section{Torsion-free nilpotent groups}\label{SN}

\subsection{Auxiliary results}\label{S:AuxRes}

%
%
%
%

The proofs of Theorems~\ref{T:Main} and \ref{T:SubMain} are obtained constructing  
specializations from the division ring of fractions of a skew polynomial ring to 
a division ring known to contain free subgroups.

The theory we use are detailed in \cite{GP15} and will be briefly recalled below.

By a \emph{specialization} from a division ring $D$ to a division ring $E$,
we mean an epimorphism $\psi\colon T \to E$ from a subring $T$ 
of $D$ such that every element of $T$ not in $\ker \psi$ is invertible in $T$. 
In particular, a pair of elements of $T$ mapping to a pair
in $D$ which generates a free subgroup in $E^{\bullet}$ will generate a free
subgroup in $D^{\bullet}$.


We shall be considering specializations to symbol and quaternion algebras, whose
definitions we recall.

Let $m$ be a positive integer which is invertible in $k$ and suppose that $k$
contains a primitive $m$-th root of unity $\theta$. Let $a,b\in k^{\bullet}$.
The \emph{symbol algebra} $\mathcal{S}_k(a,b,\theta)$ is the $k$-algebra generated
by $\mathbf{i}$ and $\mathbf{j}$ subject to
\[ \mathbf{i}^m=a, \quad \mathbf{j}^m=b, \quad 
\mathbf{j}\mathbf{i}=\theta\mathbf{i}\mathbf{j}.\]
The algebra $S_k(a,b,\theta)$ is central simple over $k$ of dimension $m^2$ with
basis $\{\mathbf{i}^r\mathbf{j}^s : 0\leq r,s \leq m-1\}$. In the case $m=2$, we
get the so called \emph{quaternion algebras}. Therefore,
the quaternion algebra $\mathcal{S}_k(a,b,-1)$ over a field $k$ of
characteristic $\ne 2$ is the $4$-dimensional
$k$-algebra with basis $\{1,\mathbf{i},\mathbf{j},\mathbf{k}\}$ with multiplication
satisfying
\[\mathbf{i}^2=a, \quad \mathbf{j}^2=b, \quad \mathbf{j}\mathbf{i} =
- \mathbf{i}\mathbf{j}=-\mathbf{k}.\]
Some symbol algebras are division rings. This will be the case in 
Proposition~\ref{P:SpecQuatSymbol} below.

Let $P$ be a prime field, that is, $P=\Q$, the field of rational number,
or $P=\mathbb{F}_p$, the field with
$p$ elements for a prime number $p$, and let $P(\lambda, X)$ be the rational function 
field in the commuting indeterminates $\lambda$ and $X$ over $P$. 
Let $\sigma$ be the $P$-automorphism
of $P(\lambda, X)$ such that $\sigma(\lambda)=\lambda$ and $\sigma(X)=\lambda X$, and consider
the skew polynomial ring $P(\lambda,X)[Y;\sigma]$ formed by skew polynomials of the form
$c_0+c_1Y+\dotsb+c_nY^n$, with $c_i\in P(\lambda, X)$, satisfying the commutation
rule $Yc=\sigma(c)Y$, for all $c\in P(\lambda, X)$. 

Our aim is to construct a specialization from the division ring of fractions $D$ of 
$P(\lambda,X)[Y;\sigma]$ to the symbol algebra $\mathcal{S}_F(a,b,\theta)$, where
$q>1$ is an integer not divisible by $p$, $\theta$ is a primitive $q$-th root of unity and
$F$ is the rational function field $P(\theta)(a,b)$ in the commuting indeterminates $a$ and $b$ over
$P(\theta)$. 

We define the domain of the specialization as follows. Let $f(\lambda)$ be the minimal
polynomial of $\theta$ over $P$. The  principal ideal $I=f(\lambda)P[\lambda]$ of 
$P[\lambda]$ 
is the kernel of the evaluation map $P[\lambda]\rightarrow \mathcal{S}_F(a,b,\theta)$, 
$\lambda\mapsto \theta$.
This homomorphism can be extended in a unique way to a homomorphism $\phi\colon 
R\rightarrow \mathcal{S}_F(a,b,\theta)$,
where $R$ stands for the localization of $P[\lambda]$ at $I$. Notice that the image 
of $\phi$
is contained in $P(\theta)$. Now, let $\varphi\colon R[X]\rightarrow 
\mathcal{S}_F(a,b,\theta)$
be the unique homomorphism coinciding with $\phi$ on $R$ such that $X\mapsto 
\mathbf{i}$. It is clear that the image of $\varphi$ is contained
in $P(\theta)(\mathbf{i})$. Because the powers of $\mathbf{i}$ are left 
$P(\theta)$-linearly independent, one can see that $\ker \phi=f(\lambda)R[X]$. 
Now, since $YX=\lambda XY$ in $R[X][Y;\sigma]$, 
there exists a unique extension of $\phi$ to a homomorphism 
$\psi\colon R[X][Y;\sigma]\rightarrow \mathcal{S}_F(a,b,\theta)$ 
sending $Y\mapsto \mathbf{j}$. We claim 
that $\ker \psi$ coincides with
the ideal of $R[X][Y;\sigma]$ generated by the central element $f(\lambda) $. Indeed, 
given $W\in \ker\psi$, since
the powers of $\mathbf{j}$ are left $P(\theta)(\mathbf{i})$-linearly independent,
as a polynomial in $Y$, $W$ must have 
all its coefficients in $\ker \varphi$. It follows that $W$ is a multiple of $f(\lambda)$, 
as desired. Note that $\ker\psi$
is a completely prime ideal of $R[X][Y;\sigma]$ and that $\bigcap_{n=0}^\infty 
(\ker\psi)^n=0$. Now \cite[Lemma~3.2]{GP15}
implies that $M\setminus \ker\psi$ is a right denominator set in $R[X][Y;\sigma]$. 
The localization
of $R[X][Y;\sigma]$ at $M$ is a local subring $T$ of $D$ and $\psi$ can be extended 
to $T$. Moreover, if $B$
is the maximal ideal of $T$ (and the kernel of $\psi$), then $T/B\cong 
\mathcal{S}_F(a,b,\theta)$.
So $\psi$ is the desired specialization.

Now, $D$ is isomorphic to the division 
ring of fractions of the group algebra $P[\Gamma]$ of the Heisenberg group $\Gamma$ 
over $P$ under a map satisfying $\lambda \mapsto [y,x], X\mapsto x, Y\mapsto y$.
Therefore, one obtains the following

\begin{proposition}\label{P:SpecQuatSymbol} 
Let $P$ be a prime field, let $q\geq 1$ be an integer not divisible by 
the characteristic of $P$ and let $\theta$ be a primitive $q$-th root of unity. 
Let $D$ be the division ring of fractions 
of the group algebra $P[\Gamma]$ of
the Heisenberg group $\Gamma$ over $P$. Then there exists a specialization $\psi$ from
$D$ to $\mathcal{S}_F(a,b,\theta)$ such that $\psi(x)=\mathbf{i}, \psi(y)=\mathbf{j}$ and 
$\psi([y,x])=\theta$, where $F$ stands for the rational function field 
$P(\theta)(a,b)$ in the commuting indeterminates $a$ and $b$ over $P(\theta)$.
\end{proposition}

Our next step is to find inside a nonabelian torsion-free nilpotent group with
an involution an invariant Heisenberg subgroup.

We shall accomplish this in Proposition~\ref{prop:InvHeis}.

\begin{lemma}\label{le:01}
Let $G$ be a torsion-free nilpotent group with center $C$. If 
$G$ has an abelian subgroup $H$, containing $C$, such that for every 
$x\in G$, there exist a positive integer $r$ such that $x^rC\in H/C$,
then $G$ is abelian. 
\end{lemma}

\begin{proof}
Given $x,y\in G$, let $r,s$ be positive integers such that $x^rC$ and $y^sC$
belong to $H/C$.
Since $H$ is abelian, it follows that $x^ry^s=y^sx^r$. But in a torsion-free
nilpotent group this implies $xy=yx$ (see, e.g. \cite[16.2.9]{KM79}). 
\end{proof}

We start by dealing with nilpotency class $2$. Recall that
in a torsion-free group, any two elements which do not commute, whose commutator 
commute with both, generate a Heisenberg subgroup. In what follows, we shall
frequently use the fact that if $G$ is a torsion-free nilpotent group with
center $C$, then $G/C$ is torsion-free (for a proof, see \cite[Lemma~11.1.3]{dP85}).

\begin{lemma}\label{le:class2}
Let $G$ be a finitely generated torsion-free nilpotent group of class $2$ with
involution $\ast$. Then $G$ contains a $\ast$-invariant Heisenberg subgroup. 
More precisely,
there exist $x,y\in G$ such that $[x,y]\ne 1$, $x^{\ast}=x^{\pm 1}, y^{\ast}=y^{\pm 1}$.
\end{lemma}

\begin{proof}
%
%
Let $C$ denote the center of $G$. It
follows from the hypothesis on the nilpotency class of $G$ that $G/C$ is a finitely 
generated torsion-free
abelian group, so there exists an isomorphism $\varphi\colon G/C\to \Z^n$, for some integer 
$n\geq 2$ (since $G$ is nonabelian). Because $C$ is a characteristic subgroup of $G$, it 
is invariant under $\ast$; so $\ast$
induces an automorphism on $G/C$ whose square is the identity. Therefore, there 
exists $A\in\GL(n,\Z)$ such that $A^2=I_n$ and 
$\varphi(g^{\ast}C)=A\varphi(gC)$,
for every $g\in G$.

Since $A$ is diagonalizable over $\Q$ with eigenvalues $1$ or $-1$, there exists 
a $\Q$-basis $\{v_1,\dotsc,v_n\}$ of $\Q^n$ with $v_i\in\Z^n$ satisfying
$A v_i =\pm v_i$, for all $i=1,\dotsc,n$. For each $i=1,\dotsc, n$,
pick $h_i\in G$ such that $\varphi(h_iC)=v_i$ and let $H$ be the subgroup of $G$ generated
by $\{h_1,\dotsc,h_n\}$. 

Given $x\in G$, let $q_1,\dotsc,
q_n\in \Q$ be such that $\varphi(xC)=q_1v_1+\dotsb +q_nv_n$. Choose a positive integer $r$
such that $rq_i=a_i\in \Z$ for all $i=1,\dotsc,n$. It follows that $x^rC=h_1^{a_1}\dotsm
h_n^{a_n}C\in H/C$. Since $G$ is nonabelian, it follows from Lemma~\ref{le:01}, 
that $H$ is nonabelian; so, we can suppose that $[h_1,h_2]\ne 1$.

By the choice of these elements, it follows that there exist $z_1,z_2\in C$ such that
$h_i^{\ast}=h_i^{\varepsilon_i}z_i$, where $\varepsilon_i\in\{-1,1\}$, for $i=1,2$.
Moreover, these central elements satisfy $z_i^{\ast}=z_i^{-\varepsilon_i}$, since
$h_i=(h_i^{\ast})^{\ast}=(h_i^{\varepsilon}z_i)^{\ast}=z_i^{\ast}(h_i^{\ast})^{\varepsilon_i}=
z_i^{\ast}(h_i^{\varepsilon_i}z_i)^{\varepsilon_i}=z_i^{\ast}z_i^{\varepsilon_i}h_i$. 

We are now in a position to define the generators of a $\ast$-invariant subgroup
of $G$. 

Let $x=h_1^{2\varepsilon_1}z_1$ and $y=h_2^{2\varepsilon_2}z_2$. Then \
\begin{align*}
x^{\ast} &= (h_1^{2\varepsilon_1}z_1)^{\ast} = z_1^{\ast}(h_1^{\ast})^{2\varepsilon_1}=
z_1^{-\varepsilon_1}(h_1^{\varepsilon_1}z_1)^{2\varepsilon_1} = 
z_1^{-\varepsilon_1}z_1^{2\varepsilon_1}h_1^{2\varepsilon_1^2}\\
&=z_1^{\varepsilon_1}h_1^2 = (h_1^{2\varepsilon_1}z_1)^{\varepsilon_1}=x^{\varepsilon_1}
\end{align*}
And similarly, $y^{\ast}=y^{\varepsilon_2}$.

It remains to prove that $x$ and $y$ do not commute. For that, we shall use the fact
that in a nilpotent group of class $2$, the commutator is a homomorphism in each entry 
(see \cite[Lemma~3.4.1]{dP85}). So, one has
$[x,y]=[h_1^{2\varepsilon_1}z_1,h_2^{2\varepsilon_2}z_2]=
[h_1,h_2]^{4\varepsilon_1\varepsilon_2}\ne 1$, for $G$ is torsion-free.

%
\end{proof}

\begin{proposition}\label{prop:InvHeis}
Let $G$ be a nonabelian torsion-free nilpotent group with involution $\ast$. Then $G$ 
contains a $\ast$-invariant Heisenberg subgroup. More precisely, 
there exist $x,y\in G$ such that $[x,y]\ne 1$, $[x,[x,y]]=[y,[x,y]]=1$, 
$x^{\ast}=x^{\pm 1}, y^{\ast}=y^{\pm 1}$.
\end{proposition}

\begin{proof}
By taking the subgroup of $G$ generated by two noncommuting elements and their
images under $\ast$, we can assume that $G$ is finitely generated.

We shall argue by induction on the nilpotency class $c$ of $G$; the case
$c=2$ having been dealt with in Lemma~\ref{le:class2}.

Suppose $c>2$ and let $C$ denote the center of $G$. Then 
$G/C$ is a nonabelian finitely generated torsion-free nilpotent group of
class $c-1$, with an involution $\ast$ induced by $\ast$. By the induction
hypothesis, there exist $x,y\in G$ such that $x^{\ast}C=x^{\varepsilon}C$,
$y^{\ast}C=y^{\eta}C$, with $\varepsilon,\eta\in\{-1,1\}$, satisfying $z=[x,y]\notin C$
and $[x,z],[y,z]\in C$.
It follows that $L=\langle x,y,C\rangle$ is a $\ast$-invariant subgroup of $G$ of 
class $\leq 3$.

If $L$ has class $2$, then the result follows from Lemma~\ref{le:class2}. 

Suppose that $L$ has class $3$. Then $[x,z]\ne 1$ or $[y,z]\ne 1$.
Say $[x,z]\ne 1$. We shall show that $K=\langle x, x^{\ast},z,z^{\ast}\rangle$
is a $\ast$-invariant subgroup of $G$ of class $2$. It will be enough to
show that $[\alpha,\beta]$ lies in the center of $K$ for all 
$\alpha,\beta\in\{x,x^{\ast},z,z^{\ast}\}$. 
For each $n\geq 1$, let $\gamma_n(L)$ denote the $n$-th term in the lower central
series of $L$. 
Now, $z=[x,y]\in \gamma_2(L)$, so $z^{\ast}\in\gamma_2(L)$, because the terms in
the lower central series are fully invariant subgroups of $G$. It follows that for every 
$\alpha\in K$ and $\beta\in\{z,z^{\ast}\}$ we have $[\alpha,\beta]\in\gamma_3(L)$, which 
is a central subgroup
of $L$, hence $[\alpha,\beta]$ is central in $K$. Finally, that $[x,x^{\ast}]=1$
follows from the fact that $x^{\ast}x^{-\varepsilon}\in C$. So $K$ is indeed a 
$\ast$-invariant subgroup of $G$ of class $2$, hence, Lemma~\ref{le:class2} applies to it.
\end{proof}

\subsection{Proof of Theorem \ref{T:Main}}

By Proposition~\ref{prop:InvHeis}, $G$ contains elements $x,y$ such that $z=[y,x]\ne 1, 
[x,z]=[y,z]=1$, satisfying one of the following conditions
\begin{enumerate}
\item $x^{\ast}=x, y^{\ast}=y$; so, clearly, $z^{\ast}=z^{-1}$;
\item $x^{\ast}=x^{-1}, y^{\ast}=y^{-1}$; hence $z^{\ast}=z^{-1}$;
\item $x^{\ast}=x, y^{\ast}=y^{-1}$ and $z^{\ast}=z$.
\end{enumerate}

Let $P$ be the prime subfield of $k$ and let $D^{\prime}$ be the subdivision ring of
$D$ generated by $x$ and $y$ over $P$. Then $D^{\prime}$ is the division ring of
fractions of the group algebra $P[\Gamma]$, where $\Gamma$ is the subgroup of $G$
generated by $x$ and $y$, which is a Heisenberg group. The involution $\ast$ of
$G$ restricts to an involution $\ast$ on $\Gamma$. So, we could use 
Theorem~\ref{T:InvolHeisenb}
to guarantee that $D^{\prime}$ and, \textit{a fortiori}, $D$ contain free symmetric
and unitary pairs.

However a simpler argument than the one used in \cite{G17} is available here. So
we reprove the fact that $D^{\prime}$ contains both free symmetric and unitary pairs.
We shall make use of Proposition~\ref{P:SpecQuatSymbol} and use the notation
$\mathcal{S}_q$ for the symbol algebra $\mathcal{S}_F(a,b,\theta)$ and
$\mathcal{H}$ for the quaternion algebra $\mathcal{S}_2$.

We consider each of the three possibilities for $\ast$ on the generators $x$ and $y$
of $\Gamma$ at a time.

\begin{enumerate}
\item $x^{\ast}=x, y^{\ast}=y, z^{\ast}=z^{-1}$.
  \begin{itemize}
  \item\textit{Free symmetric pairs}. Set $u=1+x$, $v=1+y$, and let $\psi$ be 
	the specialization from $D^{\prime}$ to the quaternion algebra $\mathcal{H}$ provided by 
  Proposition~\ref{P:SpecQuatSymbol}, that is, $\psi$ satisfies $\psi(x)=\mathbf{i},
  \psi(y)=\mathbf{j}$ and $\psi(z)=-1$. Then $\psi(u)=1+\mathbf{i}, \psi(v)=1+\mathbf{j}$, 
  and, by \cite[Theorem~2]{GMS99}, $\{u, v\}$ is a free symmetric pair.
  \item\textit{Free unitary pairs.} Let $q>1$ be an integer not divisible by the
	characteristic of $P$  and let $\psi$ be the specialization from $D^{\prime}$
	to $\mathcal{S}_q$ from Proposition~\ref{P:SpecQuatSymbol}. So we have 
	$\psi(x)=\mathbf{i}$, $\psi(y)=\mathbf{j}$ and $\psi(z)=\theta$ and the relations
	$ \mathbf{i}^q=a, \mathbf{j}^q=b, 
	\mathbf{j}\mathbf{i}=\theta\mathbf{i}\mathbf{j}, \theta^q=1$ hold in $\mathcal{S}_q$.
	Set 
	\[u=\bigl(1-zx\bigr)\bigl(1-z^{-1}x\bigr)^{-1} 
	\quad\text{and}\quad
	v=\bigl(1-zy\bigr)\bigl(1-z^{-1}y\bigr)^{-1}.\]
	So, $\psi(u)=\frac{1-\theta\mathbf{i}}{1-\theta^{q-1}\mathbf{i}}$ 
	and $\psi(v)=\frac{1-\theta\mathbf{j}}{1-\theta^{q-1}\mathbf{j}}$.
  Using the right regular representation with respect to the basis $\{1, \mathbf{j}, 
	\dotsc, \mathbf{j}^{q-1} \}$ of $\mathcal{S}_q$ over the field 
	$P(\theta, a, b, \mathbf{i})$, we write $\psi(u)$ and $\psi(v)$ as matrices $A$ and $B$,
	respectively. Let $S$ be the ring of integers of $P(\theta, a, b)(\mathbf{i})$ over 
	$P(\theta,b)[a]$.
	If $\nu$ denotes the nonarchimedean valuation determined by the maximal 
	ideal of $S$ containing $1-\theta\mathbf{i}$, it is easy to verify that $A$ and $B$ satisfy 
	the conditions of \cite[Proposition~2.4]{GL14}. Therefore $A$ and $B^{-1}AB$ 
	generate a free subgroup, and so $\{u, v^{-1}uv\}$ is a free unitary pair in $D^{\prime}$. 
 \end{itemize}

\item $x^{\ast}=x^{-1}, y^{\ast}=y^{-1}, z^{\ast}=z^{-1}$.
 
Use the same pairs as in the proof of \cite[Theorem~1.1]{FG15}.

\item  $x^{\ast}=x, y^{\ast}=y^{-1}, z^{\ast}=z$.

  Let $\psi$ be the specialization from $D^{\prime}$ onto the quaternion algebra 
	$\mathcal{H}$ such that
	$\psi(x)=\mathbf{i}, \psi(y)=\mathbf{j}, \psi(z)=-1$, afforded by 
	Proposition~\ref{P:SpecQuatSymbol}.
  \begin{itemize}
  \item\textit{Free symmetric pairs.} Set $u=1+x$, $r=xy^{5}-y^{-5}x$ and $v=(1-r)(1+r)^{-1}$.
	Then $u^{*}=u$, $r^{*}=-r$ and $v^{*}=v^{-1}$. Arguing as in the proof 
	of \cite[Theorem~1.1]{FG15} for the existence of free symmetric pair, case (III)(i), 
	we conclude that  $\{u, v^{-1}uv\}$ is a free symmetric pair.
  \item \emph{Free unitary pairs.} Set $u=xy^{5}-y^{-5}x$, $v=y-y^{-1}$, $r=(1-v)(1+v)^{-1}$
	and $s=(1-u)(1+u)^{-1}$. Arguing as in the proof 
	of \cite[Theorem~1.1]{FG15} for the existence of free unitary pairs, case (III)(i), we 
	can show that $\psi(r)$ and $\psi(s^{-1}rs)$ generate a free subgroup. 
	Therefore $\{r, s^{-1}rs\}$ is a free unitary pair.
 
 \end{itemize} 
\end{enumerate}

\subsection{Proof of Theorem \ref{T:SubMain}}

Since $G$ is a torsion-free nilpotent-by-finite group such that $G^{\ast}=G$, 
it follows that $G$ contains a nilpotent subgroup $H$ such that $[G: H] < \infty$. 
Let $H^{\ast}$ denote the image of $H$ under $\ast$ inside $G$. Since 
$[G: H^{\ast}] < \infty$, it follows, by Poincar\'e's theorem on intersections of
subgroups of finite index, that 
$[G: H\cap H^{\ast}] < \infty$. 
But $G$ is not abelian-by-finite, so $H\cap H^{\ast}$ is nilpotent, nonabelian and 
invariant under $\ast$. Therefore, substituting $G$ by $H\cap H^{\ast}$, if necessary, 
we can assume that $G$ is nonabelian, torsion-free, nilpotent and invariant 
under $\ast$.

By Proposition~\ref{prop:InvHeis}, $G$ contains a $\ast$-invariant Heisenberg subgroup 
$\Gamma = \langle x, y, z: [y,x]=z\ne 1, [x,z]=[y,z]=1 \rangle$ in which one of the following
holds:
\begin{enumerate}
\item $x^{\ast}=x, y^{\ast}=y$ and $z^{\ast}=z^{-1}$;
\item $x^{\ast}=x, y^{\ast}=y^{-1}$ and $z^{\ast}=z$;
\item $x^{\ast}=x^{-1}, y^{\ast}=y^{-1}$ and $z^{\ast}=z^{-1}$.
\end{enumerate}

We have two possibilities: either $z$ is transcendental over the prime field $P$ of
$k$ or $z$ is algebraic over $P$. In this last case, $P=\Q$, because if $z$ where
algebraic over a finite field, it would have finite order, but it is contained 
in the torsion-free group $G$.

\textit{Suppose $z$ is transcendental over $P$}. In this case, it has been
shown in the proof of \cite[Theorem~1.3]{GP15} that the group algebra 
$P[\Gamma]$ embeds into $D$, and we can follow the steps trodden in the past theorem.
Let $D^{\prime}$ be the subdivision ring of $D$ generated by $P[\Gamma]$
and let $\psi$ be the usual specialization from $D^{\prime}$ to the quaternion 
algebra $\mathcal{H}$ satisfying $\psi(x)=\mathbf{i}, \psi(y)=\mathbf{j}$ and 
$\psi(z)=-1$.
We consider each of the three possibilities for $\ast$ on the generators separately:
\begin{enumerate}
\item $x^{\ast}=x, y^{\ast}=y$ and $z^{\ast}=z^{-1}$. Let $u=1+x$ and $v=1+y$.
Since $x, y \in G$, both $s=uy^{-1}u^{-1}y$ and $t=vx^{-1}v^{-1}x$ belong to 
$N\cap D^{\prime}$, and
we have $\psi(s)=(1-a)^{-1}(1+\mathbf{i})^{2}$,  $\psi(t)=(1-b)^{-1}(1+\mathbf{j})^{2}$.
Since $N$ is invariant under $\ast$, it follows that $ss^{\ast}, tt^{\ast} \in N$.
Now, by \cite[Proposition 16]{GMS99}, it follows that 
$\psi(ss^{\ast})=(1-a)^{-2}(1+\mathbf{i})^{4}$ and 
$\psi(tt^{\ast})=(1-b)^{-2}(1+\mathbf{j})^{4}$ generate a free subgroup; 
so, $\{ss^{\ast}, tt^{\ast}\}$ is a free symmetric pair.

\item $x^{\ast}=x$, $y^{\ast}=y^{-1}$ and $z^{\ast}=z$. Let $u=1+x$ and $v=1+y+y^{-1}$. 
We have $\psi(u)=1+\mathbf{i}$ and $\psi(v)=1+\beta\mathbf{j}$, where $\beta=1+b^{-1}$.
Defining  $s=uyu^{-1}y^{-1}$ and $t=vxv^{-1}x^{-1}$, we get 
$\psi(s)=(1-a)^{-1}b^{-1}(1+\mathbf{i})\mathbf{j}(1-\mathbf{i})\mathbf{j}$ and, so, 
$\psi(ss^{\ast})=(1-a)^{-2}(1+\mathbf{i})^{4}$. Similarly 
$\psi(tt^{\ast})=(1-\beta^{2}b)^{-2}(1+\beta\mathbf{j})^{4}$.
By \cite[Theorem~2]{GMS99}, $\psi(ss^{\ast})$ and $\psi(tt^{\ast})$ generate a free subgroup. 
So the same is true of the symmetric elements $ss^{\ast}$ and $tt^{\ast}$ in $N$. 

\item $x^{\ast}=x^{-1}$, $y^{\ast}=y^{-1}$ and $z^{\ast}=z^{-1}$.
Let $u=1+x+x^{-1}$ and $v=1+y+y^{-1}$, and define $s=yuy^{-1}u^{-1}$ and $t=xvx^{-1}v^{-1}$. 
Using the abbreviations $\alpha=1+a^{-1}$ and $\beta=1-b^{-1}$, we have 
\[
\psi(ss^{\ast})=\frac{(1-\alpha\mathbf{i})^{4}}{(1-\alpha^{2}a)^{2}}
\quad\text{and}\quad
\psi(tt^{\ast})=\frac{(1-\beta\mathbf{j})^{4}}{(1-\beta^{2}b)^{2}}.
\]
From \cite[Theorem 2]{GMS99}, it follows that $\{\psi(ss^{\ast}), \psi(tt^{\ast})\}$ is a free 
pair, and consequently $\{ss^{\ast}, tt^{\ast}\}$ is a free symmetric pair in $N$.
\end{enumerate}

\textit{Suppose $z$ is algebraic over $\Q$.} Here, the construction in
the proof of \cite[Theorem~1.3]{GP15} provides a specialization $\psi$ from $D^{\prime}$
to $\mathcal{H}$ such that $\psi(x^{2^n})=\mathbf{i}, \psi(y)=\mathbf{j}$ and
$\psi(z^{2^n})=-1$, for a suitable positive integer $n$. Substituting $x$ by $x^{2^{n}}$ in 
the formulas occurring in the case of $z$ transcendental over $P$, considered above,
and acting accordingly, we obtain our free symmetric pairs.

\section{Residually torsion-free nilpotent groups}\label{SR}

In this section we show how to extend the results in the previous
section to residually torsion-free nilpotent groups through
an argument, used previously in \cite{FGS13} and \cite{S11}, involving the 
Malcev-Neumann series ring defined
by an orderable group. 

\subsection{Preliminaries}\label{sec:preliminaries}



The following definitions and comments are from
\cite[Chapter~1]{Passman1}.

Let $R$ be a ring and let $G$ be a group. A \emph{crossed product}
of $G$ over $R$ is an associative ring $R[\overline{G};\sigma,\tau]$
which contains $R$ as a subring and it is
a free left $R$-module with $R$-basis the set $\overline{G}$, a copy
of $G$. Thus each element of $R[\overline{G};\sigma,\tau]$ is uniquely
expressed as a finite sum $\sum\limits_{x\in G}r_x\overline{x}$, with $r_x\in
R$. Addition is component-wise and multiplication is determined by
the two rules below. For $x, y \in G$ we have a \emph{twisting}
\[
\overline{x}\,\overline{y} = \tau(x,y)\overline{xy},
\]
where $\tau$ is a map from $G\times G$ to the group of units $U(R)$ of $R$. Furthermore,
for $x\in G$ and $r\in R$, we have an \emph{action}
\[
\overline{x}r=r^{\sigma(x)}\overline{x},
\]
where $\sigma\colon G\to \Aut(R)$ and $r^{\sigma(x)}$
denotes the image of $r$ by $\sigma(x)$.

If $d\colon G\to U(R)$ assigns to each element $x\in G$ a
unit $d_x$ of $R$, then $\widetilde{G} = \{ \widetilde{x} = d_x\overline{x} : 
x\in G \}$ is another $R$-basis for $R[\overline{G};\sigma,\tau]$.
Moreover there exist maps $\sigma'\colon G\to \Aut(R)$ and
$\tau'\colon G\times G\to U(R)$ such that
$R[\overline{G};\sigma,\tau]=R[\widetilde{G};\sigma',\tau']$. We call this a
\emph{diagonal change of basis}. Via a diagonal change of basis, if
necessary, we will assume that $1_{R[\overline{G};\sigma,\tau]}=
\overline{1}$. The embedding of $R$ into $R[\overline{G};\sigma,\tau]$ is then
given by $r\mapsto r\overline{1}$. Note that $\overline{G}$ acts on both
$R[\overline{G};\sigma,\tau]$ and $R$ by conjugation and that for $x\in
G$ and $r \in R$, we have $\overline{x}r\overline{x}^{-1} = r^{\sigma(x)}$.

If there is no action or twisting, that is, if $\sigma(x)$ is the
identity map and $\tau(x,y)= 1$ for all $x,y \in G$, then
$R[\overline{G};\sigma,\tau]= R[G]$, the usual \emph{group ring}.

If $N$ is a subgroup of $G$, the crossed product
$R[\overline{N};\sigma_{\mid N},\tau_{\mid N\times N}]$ embeds into
$R[\overline{G};\sigma,\tau]$, and we will denote it by
$R[\overline{N};\sigma,\tau]$.

Suppose that $N$ is a normal subgroup of $G$ and let
$S=R[\overline{N};\sigma,\tau]$. Then $R[\overline{G};\sigma,\tau]$ can be
seen as a crossed product of $G/N$ over $S$. More precisely
$R[\overline{G};\sigma,\tau]=S[\overline{G/N};\widetilde{\sigma},\widetilde{\tau}]$,
for appropriate $\widetilde{\sigma}$ and $\widetilde{\tau}$. 
Below we will show in detail how this is done for a group ring $R[G]$.

\medskip

A group $G$ is  \emph{orderable} if there exists a total order $<$
in $G$ such that, for all  $x, y, z \in G$,
\[
 x < y \quad\text{implies that}\quad zx < zy \text{ and } xz < yz.
\]
In this event we say that $(G,<)$ is an \emph{ordered group}.
For further details on ordered groups
the reader is referred to any monograph on the subject such as
\cite{Fuchs}. 

Let $R$ be a ring and let $(G,<)$ be an ordered group. Suppose that
$R[\overline{G};\sigma,\tau]$ is a crossed product of $G$ over $R$.
We define a new ring,
denoted $R((\overline{G};\sigma,\tau,<))$ and called \emph{Malcev-Neumann
series ring}, in which $R[\overline{G};\sigma,\tau]$ embeds. As a set,
\[
R((\overline{G};\sigma,\tau,<))=\Bigl\{f=\sum_{x\in G}a_x\overline{x} :  a_x\in R,\ \supp(f) 
\text{ is well-ordered}\Bigr\},
\] where $\supp(f)=\{x\in G\mid a_x\neq 0\}$.

Addition and multiplication are defined extending the ones in
$R[\overline{G};\sigma,\tau]$. That is, given $f=\sum\limits_{x\in G}a_x\overline{x}
$ and $g=\sum\limits_{x\in G}b_x\overline{x}$, elements of
$R((\overline{G};\sigma,\tau,<))$, one has
\[
f+g=\sum_{x\in G}(a_x+b_x)\overline{x} \quad\text{and}\quad
fg=\sum_{x\in G}\Bigl(\sum_{yz=x}a_yb_z^{\sigma(y)}\tau(y,z)\Bigr)\overline{x}.
\]
When the crossed product is a group ring $R[G]$, its Malcev-Neumann series
ring will be denoted by $R((G;<))$.

In general,  for any ring $R$ and  $f\in R((\overline{G};\sigma,\tau,<))$,   
if the
coefficient of $\min\supp f$  is an invertible element of $R$, then $f$ is
invertible. Hence, if $R$ is a division ring, then
$R((\overline{G};\sigma,\tau,<))$ will be a division ring (as proved in \cite{Malcev}
and \cite{Neumann}). 
If $R$ is a domain, then $R((\overline{G};\sigma,\tau,<))$ is a domain, and
$f\in R((\overline{G};\sigma,\tau,<))$ is invertible if and only if the
coefficient of $\min\supp f$ is invertible in $R$. 

If $K$ is a division ring, the division subring of $K((\overline{G};\sigma,\tau,<))$
generated by $K[\overline{G};\sigma,\tau]$ will be
called the \emph{Malcev-Neumann division ring of fractions} of
$K[\overline{G};\sigma,\tau]$ and denoted by $K(\overline{G};\sigma,\tau)$. It
is important to observe the following. For a subgroup $H$ of $G$,
$K((\overline{H};\sigma,\tau))$ and $K(\overline{H};\sigma,\tau)$ can be regarded
as division subrings of  $K((\overline{G};\sigma,\tau))$ and
$K(\overline{G};\sigma,\tau)$, respectively, in the natural way. In the
case of the group ring $K[G]$, the Malcev-Neumann division ring of
fractions is denoted by $K(G)$. We remark that
$K(\overline{G};\sigma,\tau)$ does not depend on the order $<$ of $G$,
see \cite{Hughes}. When the crossed product $K[\overline{G};\sigma,\tau]$ is 
an Ore domain, then $K(G;\sigma,\tau)$ is the Ore ring of quotients of $K[\overline{G};\sigma,\tau]$.

\subsection{The series argument}

Let $k$ be a field, $G$ a group and $N$ a normal subgroup of $G$. Then the group ring $k[G]$
can be regarded as a crossed product $k[N][\overline{{G}/{N}};\sigma,\tau]$ of the group $G/N$ over
the ring $k[N]$. More precisely, for each $1\neq\alpha\in G/N,$ let $x_\alpha\in G$ be a fixed
representative of the class $\alpha$. For the class $1\in G/N$
choose $x_1=1\in G$. Set $\overline{\alpha}=x_\alpha$ for each $\alpha\in
G/N$. Then $G$ can be written as a disjoint union $G=\bigcup\{N\overline{\alpha} : 
\alpha\in G/N\}$. Define
$\overline{G/N}=\{\overline{\alpha} :  \alpha\in G/N\}\subseteq G$.
This shows that
$k[G]=\bigoplus\limits_{\alpha\in G/N}k[N]\overline{\alpha}$.
Therefore $\overline{G/N}$ is a $k[N]$-basis for
$k[G]$. Moreover, we have maps
\[
\sigma\colon G/N \longrightarrow \Aut(k[N]) \quad\text{and}\quad
\tau\colon G/N\times G/N\longrightarrow U(k[N])
\]
such that $\sigma(\alpha)(y)=x_{\alpha}yx_{\alpha}^{-1}$,
for all $\alpha\in G/N$ and $y\in k[N]$, and $\tau(\alpha,\beta)=n_{\alpha\beta}$ 
is the unique element in $N$ such that
$n_{\alpha\beta}x_{\alpha\beta}=x_\alpha x_\beta$, for all $\alpha,\beta\in G/N$. 
Then,
\[
\overline{\alpha}\Bigl(\sum_{n\in N}a_nn\Bigr)=
\Bigl(\sum_{n\in N}a_nn\Bigr)^{\sigma(\alpha)}\overline{\alpha}
\quad\text{and}\quad
\overline{\alpha}\overline{\beta}=\tau(\alpha,\beta)\overline{\alpha\beta},
\]
as desired.

Notice that the assumption $x_1=1$ implies that $k[N]$ embeds in
$k[N][\overline{G/N};\sigma,\tau]$ via $an\mapsto an\cdot 1.$

The next result shows that writing a group algebra as a crossed product
over a quotient group allows the construction of an algebra homomorphism
onto the Malcev-Neumann series ring of the quotient group, in case it is orderable.
This will be used in order to pull back results known to be valid 
over ``simpler'' groups to more general situations.

%
 %

\begin{proposition}\label{prop:homomorphismseries}
Let $k$ be a field, let $G$ be a group and let $N$ be a normal subgroup of $G$ such 
that $G/N$ is orderable. Suppose that $<$ is a total order in $G/N$ such that 
$(G/N,<)$ is an ordered group. Let $\{x_\alpha\}_{\alpha\in G/N}$ be a 
left transversal of $N$ in $G$ with $x_1=1$ and regard the group ring
$k[G]$ as a crossed product $k[N][\overline{G/N};\sigma,\tau]$. Let 
$\varepsilon_N\colon k[N]\to k$ denote the usual augmentation map. 
Then the map $\Phi_N\colon k[N]((\overline{G/N};\sigma,\tau,<))\to 
k((G/N;<))$ defined by
\[
\Phi_N\Bigl(\sum_{\alpha\in G/N}f_\alpha \overline{\alpha}\Bigr)=
\sum_{\alpha\in G/N}\varepsilon_N(f_\alpha)\alpha,
\]
where $f_\alpha\in k[N]$, for each $\alpha\in G/N$, is a $k$-algebra homomorphism.
\end{proposition}

\begin{proof}
The map $\Phi_N$ is well defined by the definition of the Malcev-Neumann series rings.
Let $f,g\in K[N]((\overline{G/N};\sigma,\tau;<))$, say
$f=\sum\limits_{\alpha\in G/N}f_\alpha\overline{\alpha}$ and 
$g=\sum\limits_{\alpha\in G/N} g_\alpha\overline{\alpha}$,
$f_\alpha, g_\alpha\in k[N]$, for each $\alpha \in G/N$. Then,
clearly, $\Phi_N(f+g)=\Phi_N(f)+\Phi_N(g)$ and
\begin{align*}
\Phi_N(fg) & =  \Phi_N\biggl(\Bigl(\sum_{\alpha\in G/N}f_\alpha\overline{\alpha}\Bigr) 
\Bigl(\sum_{\beta\in G/N}g_\beta\overline{\beta}  \Bigr) \biggr) \\
           & =  \Phi_N\biggl( \sum_{\gamma\in G/N} \Bigl(\sum_{\alpha\beta=\gamma} f_\alpha 
					g_\beta^{\sigma(\alpha)}
            \tau(\alpha,\beta)\Bigr)\overline{\gamma} \biggr) \\
           & =  \sum_{\gamma\in G/N}\varepsilon_N\Bigl(\sum_{\alpha\beta=\gamma} f_\alpha 
					g_\beta^{\sigma(\alpha)}
           \tau(\alpha,\beta) \Bigr)\gamma.           
\end{align*}
Since $\tau(\alpha,\beta)=n_{\alpha\beta}$ is the unique element in $N$ 
such that $x_\alpha x_\beta=n_{\alpha\beta}x_{\alpha\beta}$, it follows that 
$\varepsilon_N(\tau(\alpha,\beta))=1$. Now, given 
$g_{\beta}=\sum_{n\in N}b_{\beta n}n$, with $b_{\beta n}\in k$ and $n\in N$, 
then \[g_\beta^{\sigma(\alpha)}=x_\alpha\Bigl(\sum\limits_{n\in N}
b_{\beta n}n\Bigr) x_\alpha^{-1}=
\sum\limits_{n\in N}b_{\beta n}x_\alpha n x_\alpha^{-1}.\] Hence,
$\varepsilon_N(g_\beta)=\varepsilon_N(g_\beta^{\sigma(\alpha)})$ and, therefore,
$\varepsilon_N\Bigl(\sum\limits_{\alpha\beta=\gamma} f_\alpha 
g_\beta^{\sigma(\alpha)}\tau(\alpha,\beta) \Bigr)=
\Bigl(\sum\limits_{\alpha\beta=\gamma} 
\varepsilon_N(f_\alpha)\varepsilon_N(g_\beta)\Bigr)$. So,
\begin{align*}
\Phi_N(fg) 
& =  \sum_{\gamma\in G/N}\Bigl(\sum_{\alpha\beta=\gamma} 
\varepsilon_N(f_\alpha)\varepsilon_N(g_\beta)\Bigr)\gamma \\
& =  \Bigl(\sum_{\alpha\in G/N}\varepsilon_N(f_\alpha)\alpha\Bigr)
\Bigl(\sum_{\beta\in G/N}\varepsilon_N(g_\beta)\beta\Bigr) \\
& =  \Phi(f)\Phi(g). \qedhere
\end{align*} 
\end{proof}

We now introduce a notion that applies to residually torsion-free nilpotent
groups.

\begin{definition}\label{def:good}
Let $k$ be a field, let $G$ be a group and suppose that the group ring $k[G]$
embeds in a division ring $D$. If $H$ is a subgroup of $G$, denote by $D_H$
the subdivision ring of $D$ generated by $k[H]$. Let $N$ be a normal subgroup
of $G$. We say that the embedding $k[G]\hookrightarrow D$ is a 
\emph{good embedding for $k[N]$} if the following conditions are satisfied:
\begin{enumerate}
\item[(i)] the group $G/N$ is orderable;
\item[(ii)] for all division rings $F$, any crossed product $F[\overline{G/N}; \sigma,\tau]$ 
is an Ore domain;
\item[(iii)] there exists a left transversal $\{x_\alpha\}_{\alpha\in G/N}$ of $N$ in $G$ 
which is left $D_N$-linearly independent. 
\end{enumerate}
\end{definition}

In the context of Definition~\ref{def:good}, it is not difficult to prove that condition 
(iii) implies that any 
left transversal $\{x_\alpha\}_{\alpha\in G/N}$ of $N$ in $G$ is
left $D_N$-linearly independent. Also, condition (iii) implies that the subring of 
$D$ generated by $D_N$ and $\{x_\alpha\}_{\alpha\in G/N}$ is a crossed product 
$D_N[\overline{G/N};\sigma,\tau]$.

Before giving examples of good embeddings, let us recall
some known results. Let $W$ be a torsion-free nilpotent group. 
By, for example, \cite[Lemma~13.1.6]{dP85}, $W$ is orderable. Moreover, 
any crossed product $F[\overline{W};\sigma,\tau]$, with $F$ a division ring,
is an Ore domain by \cite[Corollary~37.11]{Passman1}, for example.
With this results in mind, we proceed to present the main examples of good embeddings
for our purposes.

\begin{example} Let $k$ be a field and let $G$ be a group such that 
the group algebra $k[G]$ is an Ore domain. Let $D$ be the Ore division
ring of fractions of $k[G]$. Then $k[G]\hookrightarrow D$ is a good
embedding for any normal subgroup $N$ of $G$ such that $G/N$ is
torsion-free nilpotent. This follows from the fact that $k[N]$
is an Ore domain for any subgroup $N$ of $G$. Hence $D_N$ is the Ore division
ring of fractions of $k[N]$, and the elements of $D_N$ are left fractions
$p^{-1}q$ with $p,q\in k[N]$, $p\neq 0$. Using the common denominator property,
it is not difficult to prove that any left transversal of $N$ in $G$ 
is left $D_N$-linearly independent (because it is left $k[N]$-linearly independent).
\end{example}

\begin{example}
Let $k$ be a field, let $(G,<)$ be an ordered group and let $k(G)$
denote the Malcev-Neumann division ring of fractions of the group
algebra $k[G]$. Then $k[G]\hookrightarrow k(G)$ is a good
embedding for any normal subgroup $N$ of $G$ such that $G/N$ is
torsion-free nilpotent. For condition (iii),
let $\{x_\alpha\}_{\alpha\in G/N}$ be a left transversal of $N$ in $G$. For any
$n,n'\in N$ and $\alpha,\alpha'\in G/N$, $nx_\alpha=n'x_{\alpha'}$ if, and only if,
$n=n'$ and $x_\alpha=x_{\alpha'}$.  Since $k(N)$, the division ring generated
by $k[N]$ inside $k(G)$, is contained in $k((N;<))\subseteq k((G;<))$,
it is not difficult to prove that the elements of the left transversal are left
linearly independent over $k((N;<))$ and, hence, over $k(N)$.
\end{example}

Given a field $k$, a group $G$ and a division ring $D$ such that $k[G]$ embeds
in $D$, let $N$ be a normal subgroup of $G$ such that the embedding $k[G]\hookrightarrow D$
is good for $k[N]$. Suppose that $G/N$ is endowed with a total order such that 
$(G/N,<)$ is an ordered group. By condition (ii),
$D_N[\overline{G/N};\sigma,\tau]$ is an Ore domain embedded in $D_G$. 
By the universal property of 
the Ore localization, $D_G$ is the Ore division ring of quotients of 
$D_N[\overline{G/N};\sigma,\tau]$.
But $D_N[\overline{G/N};\sigma,\tau]$ also embeds in the division ring 
$D_N((G/N;\sigma,\tau,<))$
and, again, by the universal property of the Ore localization $D_N(G/N;\sigma,\tau)$ is the
Ore division ring of quotients of $D_N[\overline{G/N};\sigma,\tau]$.
Therefore $D_G$ can be identified with
$D_N(G/N;\sigma,\tau)$ inside the Malcev-Neumann series ring $D_N((G/N;\sigma,\tau,<))$. 
 
Next, we show that good embeddings provide an appropriate setting for
pulling back free pairs from division rings of fractions of torsion-free
nilpotent group algebras to more general division rings.

\begin{theorem}\label{theo:seriesargument}
Let $G$ be a group with an involution $\ast$ and let $N$ be a $\ast$-invariant
normal subgroup of $G$ such that $G/N$ is a nonabelian torsion-free nilpotent group. 
Let $k$ be a field of characteristic different from $2$ and let $k[G]\hookrightarrow D$ 
be a good embedding for $k[N]$ 
of the group algebra $k[G]$
in the division ring $D$. Suppose that the induced $k$-involution $\ast$ on $k[G]$
can be extended to a $k$-involution in $D$. 
Then $D_G$ contains a free symmetric pair with respect to $\ast$.
\end{theorem}

\begin{proof}
First observe that $\ast$ defines a $k$-involution on $k[N]$ by restriction. 
It also induces an involution $\ast$ on the quotient group $G/N$, and, therefore,
a $k$-involution on $k[G/N]$. Moreover this can be extended
to an involution $\ast$ on $k(G/N)$, because $k[G/N]$ is an Ore domain.

Let $\{x_\alpha\}_{\alpha\in G/N}$ be a 
left transversal of $N$ in $G$ with $x_1=1$. 
Regard the group algebra
$k[G]$ as a crossed product $k[N][\overline{G/N};\sigma,\tau]$ and let
$\varepsilon_N\colon k[N]\to k$ denote the augmentation map.
Observe that, for each $\alpha\in G/N$, there exists a unique
$n_\alpha\in N$ such that ${\overline{\alpha}}^{\ast}=x_\alpha^{\ast}=
n_\alpha x_{\alpha^{\ast}}=n_\alpha\overline{\alpha^{\ast}}$.  
Take $\sum\limits_{\alpha\in G/N}f_\alpha \overline{\alpha}\in k[G]$ with 
$f_\alpha\in k[N]$. Then
\[
\Bigl(\sum_{\alpha\in G/N}f_\alpha \overline{\alpha}\Bigr)^{\ast}=
\sum_{\alpha\in G/N}{\overline{\alpha}}^{\ast}f_\alpha^{\ast}=
\sum_{\alpha\in G/N}n_\alpha ({f_\alpha^{\ast}})^{\sigma(\alpha^{\ast})}
\overline{\alpha^{\ast}}.
\]
So, if $f_{\alpha}\in k$, for all $\alpha\in G/N$, we have 

\begin{equation}\label{eq:involutiononcrossed}\tag{$\dagger$}
\begin{split}
\Bigl(\sum_{\alpha\in G/N}f_\alpha \overline{\alpha}\Bigr)^{\ast} & =
\sum_{\alpha\in G/N}n_\alpha({f_\alpha^{\ast}})^{\sigma(\alpha^{\ast})}
\overline{\alpha^{\ast}}=\sum\limits_{\alpha\in G/N}n_\alpha f_\alpha 
\overline{\alpha^{\ast}}\\
&=\sum\limits_{\alpha\in G/N} 
f_\alpha n_\alpha\overline{\alpha^{\ast}}. 
\end{split}
\end{equation}

Consider, now, the homomorphism 
$\Phi_N\colon k[N]((\overline{G/N};\sigma,\tau,<))\to k((G/N;<))$, given in 
Proposition~\ref{prop:homomorphismseries}, defined by 
\[
\Phi_N\Bigl(\sum_{\alpha\in G/N}f_\alpha \overline{\alpha}\Bigr)=
\sum_{\alpha\in G/N}\varepsilon_N(f_\alpha)\alpha,
\]
where $f_\alpha\in k[N]$, for each $\alpha\in G/N$.

Suppose that $A$ is a symmetric nonzero element in $k(G/N)$ with respect to $\ast$. 
Then $A=p_Aq_A^{-1}$ for some $p_A,q_A\in k[G/N]\setminus\{0\}$. Write 
$p_A=\sum_{\alpha\in G/N}a_\alpha \alpha$ and $q_A=\sum_{\alpha\in G/N}b_\alpha\alpha$,
with $a_\alpha,b_\alpha\in k$ for each $\alpha\in G/N$. 
Define $P_A=\sum_{\alpha\in G/N}a_\alpha\overline{\alpha}$ and $Q_A=\sum_{\alpha\in G/N}
b_\alpha\overline{\alpha}\in k[G]$.
Note that $P_A$ and $Q_A$ are invertible elements in $k[N]((G/N;\sigma,\tau,<))$ because the 
least element in the support of $P_A$ and $Q_A$ is invertible in $k[N]$ since it 
belongs to $k\setminus\{0\}$.

Consider also $P_A^{\ast}$ and $Q_A^{\ast}$ as elements of $k[N]((G/N;\sigma,\tau,<))$.
By \eqref{eq:involutiononcrossed}, the least elements in the supports of 
$P_A^{\ast}$ and $Q_A^{\ast}$ are of the form $a_\alpha n_\alpha$ for some
$n_\alpha\in N$ and nonzero $a_\alpha\in k$. 
Hence $P_A^{\ast}$ and $Q_A^{\ast}$ are also invertible elements in $k[N]((G/N;\sigma,\tau,<))$.

We have 
\begin{align*}
\Phi_N(P_A) & =  \Phi_N\Bigl( \sum_{\alpha\in G/N}a_\alpha \overline{\alpha} \Bigr) 
              =  \sum_{\alpha\in G/N} \varepsilon_N(a_\alpha)\alpha \\
            & =  \sum_{\alpha\in G/N} a_\alpha\alpha 
              =  p_A.
\intertext{and}
\Phi_N(P_A^{\ast}) & =  \Phi_N\biggl(\Bigl(\sum_{\alpha\in G/N}a_\alpha 
\overline{\alpha}\Bigr)^{\ast}\biggr)
                =  \Phi_N\Bigl(\sum_{\alpha\in G/N}a_\alpha n_\alpha 
\overline{\alpha^{\ast}} \Bigr)\\
              & =  \sum_{\alpha\in G/N} \varepsilon(a_\alpha n_\alpha)\alpha^{\ast}
               =  \sum_{\alpha\in G/N} a_\alpha \alpha^{\ast}
               =  p_A^{\ast}. 
\end{align*}
Similarly, $\Phi_N(Q_A)=q_A$ and $\Phi_N(Q_A^{\ast})=q_A^{\ast}$. Since $\Phi_N$ is a ring homomorphism, it follows that $\Phi_N(Q_A^{-1})=q_A^{-1}$ and
$\Phi_N\bigl((Q_A^{\ast})^{-1}\bigr)=(q_A^{\ast})^{-1}$. 

Consider now the symmetric element $X=P_AQ_A^{-1}(Q_A^{\ast})^{-1}P_A^{\ast}\in D_G$. Then 
\begin{align*}
\Phi_N(X) & =  \Phi_N\bigl(P_AQ_A^{-1}(Q_A^{\ast})^{-1}P_A^{\ast}\bigr) \\
          & =  \Phi_N(P_A)\Phi_N(Q_A^{-1})\Phi_N\bigl((Q_A^{\ast})^{-1}\bigr)
					\Phi_N(P_A^{\ast}) \\
          & =  p_Aq_A^{-1}(q_A^{\ast})^{-1}p_A^{\ast} \\
          & =  p_Aq_A^{-1}(p_Aq_A^{-1})^{\ast} \\
          & =  AA^{\ast} \\
          & =  A^2         
\end{align*}
Since $P_A,P_A^{-1},P_A^{\ast},(P_A^{\ast})^{-1},
Q_A,Q_A^{-1},Q_A^{\ast},(Q_A^{\ast})^{-1}\in k[N]((G/N;\sigma,\tau,<))$, 
if follows that 
$X^{-1}\in k[N]((G/N;\sigma,\tau,<))$. Thus $\Phi_N(X^{-1})=A^{-2}$.

Now let $A,B$ be symmetric elements of $k(G/N)$ generating a free group, which exist
by Theorem~\ref{T:Main}. Clearly, $A^2,B^2$ also generate a free group. By the foregoing argument, there exist $X,Y \in D_G$
such that $\Phi_N(X)=A^2$, $\Phi_N(Y)=B^2$, $\Phi_N(X^{-1})=A^{-2}$ and $\Phi_N(Y^{-1})=B^{-2}$. 
Therefore $X,Y$ also generate a free group.
\end{proof}

\subsection{Proof of Theorem~\ref{th:restorfrnil}}

Here we shall apply the main result of the preceding section in order
to extend Theorem~\ref{T:Main} to the class of residually torsion-free
nilpotent groups.

\begin{definition}
A group $G$ is \emph{residually torsion-free nilpotent} if for
each $g\in G$, there exists a normal subgroup $N_g$ of $G$ such that
$g\notin N_g$ and $G/N_g$ is torsion-free nilpotent.
\end{definition}

Let $G$ be a group. If $H$ is a subgroup of $G$, we denote by $\sqrt{H}$ the subset of $G$ defined by
\[\sqrt{H}=\{x\in G : x^n\in H \text{, for some } n\geq 1\}.\]

For the next result, given a group $G$ we shall denote the $n$-th
term in its lower central series by $\gamma_n(G)$. That is,
we set $\gamma_1(G)=G$ and, for $n\ge 1$, define 
$\gamma_{n+1}(G)=[\gamma_n(G),G]$.

\begin{lemma}
Let  $G$ be a residually torsion-free nilpotent group. Then the following
hold true.
\begin{enumerate}
\item For each $n\geq 1$, $\sqrt{\gamma_n(G)}$ is a normal subgroup of $G$
such that $G/\sqrt{\gamma_n(G)}$ is torsion-free nilpotent.\label{item:1}
\item The chain of normal subgroups of $G$ \label{item:2}
\[
G=\sqrt{\gamma_1(G)}\supseteq \sqrt{\gamma_2(G)}\supseteq \dotsb
\supseteq \sqrt{\gamma_n(G)}\supseteq\dotsb
\]
 satisfies
 \begin{enumerate}
 \item $\bigl[G,\sqrt{\gamma_{n}(G)}\bigr]\subseteq \sqrt{\gamma_{n+1}(G)}$, 
for all $n\geq 1$, and 
\item $\sqrt{\gamma_n(G)}/\sqrt{\gamma_{n+1}(G)}$ is a torsion-free abelian group, 
for all $n\geq 1$. 
 \end{enumerate}
\item $\bigcap_{n\geq 1}\sqrt{\gamma_n(G)}=\{1\}$. \label{item:3}
\item $G$ is an orderable group. \label{item:4}
\end{enumerate}
\end{lemma}
\begin{proof}
\eqref{item:1} and \eqref{item:2} follow from \cite[Lemma~11.1.8]{dP85}. For
\eqref{item:3}, given $g\in G$, since $G/N_g$ is nilpotent, there exists $n_g\geq 1$ 
such that $\gamma_{n_g}(G)\subseteq N_g$. Moreover, since $G/N_g$ is a torsion-free
group, $\sqrt{\gamma_{n_g}(G)}\subseteq N_g$.
Hence,
\[\bigcap_{n\geq 1}\gamma_n(G)\subseteq\bigcap_{n\geq
1}\sqrt{\gamma_n(G)}\subseteq \bigcap_{g\in G}
\sqrt{\gamma_{n_g}(G)}\subseteq \bigcap_{g\in G}N_g=\{1\}.\]
Now,
\eqref{item:4} follows from \eqref{item:1}, \eqref{item:2}, \eqref{item:3},
in view of \cite[Theorem~IV.6]{Fuchs}.
\end{proof}

We are ready to present a proof of Theorem~\ref{th:restorfrnil}.

\begin{proof}[Proof of Theorem~\ref{th:restorfrnil}]
The $k$-involution on the group algebra $k[G]$ induced by the involution
$\ast$ on $G$ extends uniquely to a $k$-involution on the Malcev-Neumann
division ring of fractions $k(G)$ of $k[G]$ (see \cite[Theorem~2.9]{FGS13}). 
Let $n\geq 2$ be fixed. 
Since $G$ is not abelian, $G/\sqrt{\gamma_n(G)}$ is
a nonabelian torsion-free nilpotent group. Apply Theorem~\ref{theo:seriesargument} with
$N=\sqrt{\gamma_n(G)}$.
\end{proof}

It should be remarked that the argument used in the proofs of 
Theorems~\ref{theo:seriesargument} and \ref{th:restorfrnil} can also 
be used to prove the existence of symmetric free group algebras. More
precisely, under the hypothesis of Theorem~\ref{theo:seriesargument}, it can
be shown that if the characteristic of $k$ is zero, then there exist
symmetric elements $X,Y\in D_G$ such that the $k$-subalgebra of $D_G$
generated by $\{X,X^{-1},Y,Y^{-1}\}$ is the free group $k$-algebra generated
by $\{X,Y\}$. For a proof, the same argument used in the proof of 
Theorem~\ref{theo:seriesargument} shows that if $A$ and $B$ are
symmetric elements in $k(G/N)$ generating a free group algebra, then
$X$ and $Y$ will be symmetric elements in $D_G$ generating a free group
algebra. That $k(G/N)$ contains such a pair follows from \cite{SII}. Similarly, one can use this version of
Theorem~\ref{theo:seriesargument} to prove that if $k$ has characteristic zero
and $G$ is a nonabelian residually torsion-free nilpotent group with an involution
$\ast$, then the Malcev-Neumann division ring of fractions $k(G)$ of $k[G]$
contains a free group $k$-algebra generated by symmetric elements with
respect to the $k$-involution on $k(G)$ induced by $\ast$.


\begin{thebibliography}{20}

\bibitem{FG15}
  V. O. Ferreira and J. Z. Gon\c calves,
	\textit{Free symmetric and unitary pairs in division rings infinite-dimensional 
	over their centers},
	Israel Journal of Mathematics~\textbf{210} (2015), no.~1, 297--321.
	
 
\bibitem{FGS13}
	V. O. Ferreira, J. Z. Gon\c calves and J. S\'anchez,
	\textit{Free symmetric group algebras in division rings generated by poly-orderable groups},
	Journal of Algebra~\textbf{392} (2013), 69--84.
	
\bibitem{Fuchs}
  L. Fuchs,
	\textit{Partially ordered algebraic systems},
	Pergamon Press, Oxford, 1963.
	
\bibitem{G17}
  J. Z. Gon\c calves,
	\textit{Free pairs of symmetric and unitary units in normal subgroups of a division ring},
	Journal of Algebra and its Applications~\textbf{16} (2017), no.~6, 1750108, 17pp.	
	
\bibitem{GL14}
  J. Z. Gon\c calves and A. I. Lichtman, 
	\textit{Free subgroups in division rings generated by group rings of soluble groups},
	International Journal of Algebra and Computation~\textbf{24} (2014), no.~8, 1127--1140.

\bibitem{GMS99}
  J. Z. Gon\c calves, A. Mandel and M. Shirvani,
	\textit{Free products of units in algebras. I. Quaternion algebras},
	Journal of Algebra~\textbf{214} (1999), no.~1, 301--316.
  
	
\bibitem{GP15}
  J. Z. Gon\c calves and D. S. Passman, 
	\textit{Explicit free groups in division rings},
	Proceedings of the American Mathematical Society~\textbf{143} (2015), no.~2, 459--468.
 
\bibitem{GS06}
  J. Z. Gon\c calves and M. Shirvani,
	\textit{Free symmetric and unitary pairs in central simple algebras with involution},
	Groups, rings and algebras, Contemporary Mathematics~\textbf{420}, American Mathematical
	Society, Providence, RI, 2006, pp.~121--139.
 
\bibitem{GS12}
  J. Z. Gon\c calves and M. Shirvani,
	\textit{A survey on free objects in division rings and in division rings with an involution},
	Communications in Algebra~\textbf{40} (2012), no.~5, 1704--1723.
	
\bibitem{Hughes}
  I. Hughes,
	\textit{Division rings of fractions for group rings},
	Communications on Pure and Applied Mathematics~\textbf{23} (1970), 181--188.
	
\bibitem{KM79}
  M. I. Kargapolov and Ju. I. Merzljakov,
	\textit{Fundamentals of the theory of groups},
	Springer-Verlag, New York-Berlin, 1979.
	
\bibitem{LR04}
  J. C. Lennox and D. J. S. Robinson,
	\textit{The theory of infinite soluble groups},
	Clarendon Press, Oxford, 2004.
 
\bibitem{Lichtman77}
  A. I. Lichtman, 
	\textit{On subgroups of the multiplicative group of skew fields},
	Proceedings of the American Mathematical Society~\textbf{63} (1977), no.~1, 15--16.
 
\bibitem{Lichtman78}
  A. I. Lichtman, 
  \textit{Free subgroups of normal subgroups of the multiplicative group of skew fields},
	Proceedings of the American Mathematical Society~\textbf{71} (1978), no.~2, 174--178. 
	
	
\bibitem{Malcev}
  A. I. Malcev,
	\textit{On the embedding of group algebras in division algebras} (Russian),
	Doklady Akad. Nauk SSSR (N.S.)~\textbf{60} (1948), 1499--1501. 
	
\bibitem{Neumann}
  B. H. Neumann,
	\textit{On ordered division rings},
	Transactions of the American Mathematical Society~\textbf{66} (1949), 202--252.
 %

\bibitem{dP85}
  D. S. Passman,
	\textit{The algebraic structure of group rings},
	Robert E. Krieger Publishing Co., Inc., Melbourne, FL, 1985.
	
\bibitem{Passman1}
  D. S. Passman,
	\textit{Infinite crossed products},
	Academic Press, Inc., Boston, MA, 1989. 
	
\bibitem{S11}	
  J. S\'anchez,
  \textit{Free group algebras in Malcev-Neumann skew fields of fractions},
  Forum Mathematicum~\textbf{26} (2014), no.~2, 443--466.

\bibitem{SII}
  J. S\'anchez,
	\textit{Free group algebras in division rings with valuation II}, 
	work in progress.


\end{thebibliography}
\end{document}